\documentclass[final]{amsart}
\usepackage{amsmath,amssymb,amsfonts,amsthm,slashed,braket}
\usepackage[all]{xy}
\usepackage{mathrsfs}
\usepackage{fancyhdr}
\usepackage{showkeys}
\usepackage{hyperref}

\title{The symmetric power and etale realisation functors commute}
\author{Arnav Tripathy}

\newcommand\nc{\newcommand}
\nc\linesep{\bigskip}
\nc\newprob[1]{\marginnote{#1}[\parskip]}
\nc\bA{\mathbb A}
\nc\bC{\mathbb C}
\nc\bD{\mathbb D}
\nc\bR{\mathbb R}
\nc\bZ{\mathbb Z}
\nc\bQ{\mathbb Q}
\nc\bP{\mathbb P}
\nc\bV{\mathbb V}
\nc\bW{\mathbb W}
\nc\bG{\mathbb G}
\nc\mf\mathfrak
\nc\mc\mathcal
\nc\mb\mathbb
\nc\brac[1]{\langle#1\rangle}
\nc\abs[1]{\lvert#1\rvert}
\nc\norm[1]{\lVert#1\rVert}
\nc\onto{\twoheadrightarrow}
\nc\into{\hookrightarrow}
\nc\lto{\longrightarrow}
\nc\action{\curvearrowright}
\nc\on\operatorname
\nc\wbar\overline
\nc\what\widehat
\nc\wtilde\widetilde
\nc\nop\DeclareMathOperator
\nc\eps{\varepsilon}
\nc\tsym{\widetilde{\text{Sym}}}
\nc\oarrow[1]{\overset{#1}\to}
\nop\Hom{Hom}
\nop\End{End}
\nop\Aut{Aut}
\nop\im{Im}
\nop\id{id}
\nop\tr{Tr}
\nop\coker{coker}
\nop\Spec{Spec}
\nop\Jac{Jac}
\nop\Ext{Ext}
\nop\Tor{Tor}
\nc\op{\text{op}}
\nop\loc{Loc}
\nop\Frac{Frac}
\nc\ann{\text{ann}}
\nop\QCoh{QCoh}
\nop\Coh{Coh}
\nop\Sym{Sym}
\nop\gr{Gr}
\nop\Tot{Tot}
\nop\Fl{Fl}
\nop\tGamma{\widetilde\Gamma}
\nop\tloc{\widetilde{\text{Loc}}}
\nop\rep{Rep}
\nop\proj{Proj}
\nc\oo[1]{\overset\circ{#1}}
\nop\ospec{\oo{Spec}}
\nop\oTot{\oo{Tot}}
\nop\Bl{Bl}
\nop\Comp{Comp}
\nop\Ho{Ho}
\nop\cone{Cone}
\nop\LKE{LKE}
\nop\RKE{RKE}
\nop\pd{pd}
\nop\cd{cd}
\nop\depth{depth}
\nop\ass{Ass}
\nop\supp{supp}
\nop\codim{codim}
\nop\holim{\underset{\lto}{holim}}
\nop\dlim{\underset{\lto}{lim}}

\nop\uHom{\underline{\Hom}}
\nop\Pic{Pic}
\nop\Cl{Cl}
\nop\Div{Div}
\nop\rank{rank}
\nop\Der{Der}
\nop\dimrel{dim.rel}
\nc\sHom{\mathscr Hom}
\nc\sExt{\mathscr Ext}
\nc\dto{\dashrightarrow}
\nop\rspec{\bf Spec}
\nop\Gal{Gal}
\nop\Ind{Ind}
\nop\Frob{Frob}
\nop\Fib{Fib}
\nop\ratdim{rat\ dim}
\nop\Mod{Mod}
\nop\rat{rat}
\nop\val{val}
\nop\Rep{Rep}
\nop\colim{colim}

\theoremstyle{theorem}
\newtheorem{thm}{Theorem}
\newtheorem{lemma}[thm]{Lemma}
\newtheorem{cor}[thm]{Corollary}

\newtheorem{prop}[thm]{Proposition}

\theoremstyle{remark}
\newtheorem{rem}[thm]{Remark}

\begin{document} 
\maketitle
\tableofcontents

\begin{abstract}

We show that under mild hypotheses on a proper algebraic space $X$, the functors of taking its symmetric powers and its \'{e}tale realisation commute up to weak equivalence. We conclude an effective version of the Dold-Thom theorem for the \'{e}tale site and discuss the stabilisation results for the natural morphisms of \'{e}tale homotopy groups $\pi_k \Sym^n X \to \pi_k \Sym^{n+1} X$ in the context of the Weil conjectures. 

\end{abstract}

\section{Introduction}

Symmetric powers of schemes are interesting for several reasons. Perhaps the most immediate is that they are the simplest example of a moduli space as the most elementary construction of a moduli space of points on a scheme. As we discuss throughout the paper, symmetric powers also have various natural appearances in number theory and algebraic topology. The multitude of appearances of this ubiquitous functor in disparate but nonetheless related subjects makes it desirable to understand how these interpretations compare with one another. In this paper, we use the notion of \'{e}tale homotopy type as developed by Artin and Mazur to pass from an algebraic category of scheme-like objects to a topological category of topological space-like objects, and under mild hypotheses, our result is the best comparison that one might hope for:

\begin{thm} Let $X$ be a proper, normal, noetherian, geometrically connected algebraic space over a separably closed field $k$. The natural map $\Sym^n (X_{et}) \to (\Sym^n X)_{et}$ of pro-homotopy types is a weak equivalence in this category. \end{thm}

We stress that this comparison result is a weak equivalence ``on the nose''; in particular, we do not need to localise away from the characteristic of the base field. Certainly, properness is necessary for the strength of this conclusion. We will discuss the notation in the above theorem statement shortly, but first we note that this result should not be considered as a collection of independent results for various $n$, Rather, as we increase $n$, the results fit together naturally in commutative diagrams. Indeed, one of the major impetuses for our interest in this statement is how it interacts with the stabilisation maps that relate $\Sym^n$ to $\Sym^{n+1}$. The following result easily follows from Theorem $1$ by functoriality:

\begin{thm} Let $X$ be an algebraic space as above with a choice of base point $x \in X(k)$ and corresponding stabilisation map $\alpha_n: \Sym^n X \to \Sym^{n+1} X$ given by including an extra copy of the base point $x$. Then we have the following commutative diagram of pro-homotopy types: \begin{equation*} \xymatrix{ \Sym^n (X_{et}) \ar[r] \ar[d] & (\Sym^n X)_{et} \ar[d] \\ \Sym^{n+1} (X_{et}) \ar[r] & (\Sym^{n+1} X)_{et}, } \end{equation*} where the horizontal arrows are weak equivalences as before.\end{thm}

%Above, we supposed for simplicity that $X$ had a rational point over $k$ to use as a base point. Similar results also hold in the case that one cannot or does not wish to pick a rational point $x$ to use as a base point; for example, if one picks $x \in X(E)$ for some field extension $E/k$ of degree $d$, the same result as above holds with the stabilisation maps now being of the form $\Sym^n X \to \Sym^{n+d} X$. We will see that the particulars of over what field extension our base point is defined will not matter; the first step of the proof is to show that we can immediately base change to the algebraic closure $\overline{k}$ and prove all the desired results there.

We note some corollaries of our main results above. First, recall the classic Dold-Thom theorem of algebraic topology, which may abstractly be written as $$\Sym^{\infty} X = (X \wedge H\mb{Z})_0,$$ i.e. that the infinite symmetric power of a space is given by the connected component of smashing with the Eilenberg-MacLane spectrum for the integers. More simply, homotopy groups of the infinite symmetric power are the reduced homology groups of the original space. In fact, Suslin and Voevodsky use this approach for the original definition of motivic homology in~\cite{VS}. As such, our comparison result indicates that different definitions of \'{e}tale homology agree, namely that the (reduced) homology of the \'{e}tale homotopy type agrees with the homotopy groups of the (algebraic) infinite symmetric power. This agreement of definitions stands in parallel with Voevodsky's demonstration in~\cite{Voevodsky} that the different definitions of motivic homology via symmetric powers or motivic complexes as above or via Bloch's higher Chow groups do in fact agree. While neither the statement nor proof of our result is as technical as Voevodsky's motivic results, we still find it foundationally satisfactory to have these comparison results firmly stated for \'{e}tale homotopy. Here then is the Dold-Thom theorem for the \'{e}tale site:

\begin{cor} For $X$ as above and any $n \ge 0$, we have $\pi_n^{et} \Sym^{\infty} X \simeq \tilde{H}_n(X_{et}, \mb{Z})$. \end{cor}

The \'{e}tale homotopy pro-groups $\pi_n^{et}$ are by definition the homotopy pro-groups of the \'{e}tale realisation $X_{et}$, which we expand on further in the next section. As for the infinite symmetric power $\Sym^{\infty} X$, one could define it as an ind-algebraic space after choosing some base point so as to define the stabilisation maps. However, due to the stabilisation result we state next, we may in practice simply use any sufficiently large symmetric power $\Sym^k X$ to obtain the correct homotopy pro-group in a given degree. Note that the above result involves integral \'{e}tale homology on the right, which does not have an immediate interpretation in terms of derived functors due to lack of enough projectives in any reasonable category; hence, what we mean by integral \'{e}tale homology is by definition the integral homology of the \'{e}tale realisation. One may use the universal coefficient theorem to derive a result with cohomology in some torsion coefficient group on the right hand side, which now does have an immediate interpretation familiar to algebraic geometers.

Furthermore, in the usual topological Dold-Thom result, not only do we know the stable limiting value of $\Sym^n X$ as $n \to \infty$, but we also have precise bounds on how quickly these finite approximants converge. As such, we expect to see a similar sort of stability phenomenon in the algebro-geometric case. In other words, as $n$ increases, we wish for the natural map $\Sym^n X \to \Sym^{n+1} X$ to become arbitrarily highly connected, just as in the topological category. Indeed, we have the following result: %, where we again choose a rational base point for simplicity:

\begin{cor} For $X$ as above with a choice of base point $x \in X(k)$ and corresponding stabilisation maps $\alpha_n$, we have that $\pi_k \alpha_n: \pi_k^{et} \Sym^n X \to \pi_k^{et} \Sym^{n+1} X$ is an isomorphism for $k < n$ and a surjection for $k = n$. \end{cor}

\begin{rem} For some remarks on strictness, see the discussion of this result in section $4$. Futhermore, we should note that these theorems are only interesting in characteristic $p$; in characteristic zero, they follow via profinitely completing the results for the complex realisation functor. In particular, in characteristic zero, we have no need of the normality or properness assumptions. \end{rem}

An intriguing question posed to us by R. Vakil is how the splitting in the abstract form of the Dold-Thom theorem may be related to other forms of splitting of interest to algebraic geometers. The spectrum $X \wedge H\mb{Z}$ has no $k$-invariants and splits completely into shifts of Eilenberg-MacLane spectra, each one of which is responsible for precisely one cohomological degree of the original total cohomology of $X$. Meanwhile, one of the main questions in the Grothendieck program of standard conjectures surrounding the investigation of the Weil conjectures is the splitting of a variety into its cohomological pieces graded by degree, now in some appropriate category of motives. The question is then to relate such conjectural motivic splitting of the algebraic object $\Sym^{\infty} X$ to the topological splitting implicit in Dold-Thom. We do not pursue the question in this paper.

For knowledgeable readers, we comment on the assumptions in Theorem $1$ and its consequences. In addition to purely technical assumptions like local noetherianness, we assume our space $X$ is proper and normal. That we need $X$ proper is not surprising as even the Kunneth formula for fundamental groups of products, $\pi_1(X \times Y) \simeq \pi_1(X) \times \pi_1(Y)$, fails in the nonproper case in positive characteristic. For example, if our result held, it would imply that $\mb{A}^2 \simeq \Sym^2 \mb{A}^1$ had abelian \'{e}tale fundamental group. However, just considering $\mb{A}^2 \simeq \mb{A}^1 \times \mb{A}^1$, we have $\pi_1(\mb{A}^1) \hookrightarrow \pi_1(\mb{A}^2)$, since a connected finite \'{e}tale cover of $\mb{A}^1$ remains so after multiplying with a trivial second $\mb{A}^1$ factor. However, $\pi_1(\mb{A}^1)$ is a large, complicated pro-$p$ group because of Artin-Schreier extensions and is highly nonabelian. So the failure of Kunneth implies that for $X$ nonproper, we have little control over the fundamental group, and we have no reason to suspect it stabilises or does anything nice except after localising away from the characteristic. Normality, on the other hand, enters our argument in a somewhat more subtle way, where the real property that often enters our arguments is geometric unibranchness. We are unsure to what extent this assumption is necessary. Finally, we restrict in this paper to the case that the ground field $k$ is separably closed; all the results should admit an immediate extension to $k$ a general field, where on the topological side, the symmetric power needs to occur fibrewise in a suitable way over the base $(\Spec k)_{et} \simeq B\Gal(k)$, and such an extension is indeed of great conceptual interest as the original motivation for the results of this paper stemmed from the Weil conjectures, for which we take $k = \mb{F}_q$. However, we are in the process of establishing a far more general relative form of the results of this paper such that the case of a general base field will fall out as a simple corollary; we ask the interested reader to investigate this forthcoming paper for the details when confronted with a more general base.

We start in the next section by defining and giving a brief discussion of the technology used to formulate our results, particularly the \'{e}tale pro-homotopy type of Artin and Mazur. We include one last introductory section in section $3$ before moving on to the proofs. This section, of interest to the arithmetically-inclined reader, describes our point of view on this paper as a homotopical refinement of the weakest information present in the Weil conjectures and asks for similar refinements of stronger asymptotics. In section $4$, we present the topological story and the appropriate stabilisation result for symmetric powers. The main result we recall here is a quantitative version of Dold-Thom which not only identifies the stable limit of the sequence of homotopy types $\Sym^n X$ but also when exactly the sequence converges in any given homotopical degree. In the latter half of the paper, we move to the algebro-geometric side of the story; having already stated the main result we wish to prove here in the introduction, we begin with some necessary reductions in section $5$ before giving a lengthy discussion of the fundamental group in section $6$. Section $6$ is the technical heart of the paper: we concretely track the category of finite \'{e}tale covers in order to replicate the straightforward argument of section $4$ in the algebro-geometric category. Finally, in section $7$, we invoke a version of the Whitehead theorem due to Artin-Mazur for the pro-category and cohomological results of Deligne to finish the argument. An appendix provides some missing details of our topological argument from section $4$ that are not crucial to the main flow of the paper but nonetheless serve to keep the paper more self-contained.

We first and foremost owe a great deal to Ravi Vakil for extensive discussions and careful comments on this paper. We furthermore thank Marc Hoyois for a crucial comment on an initial draft of this paper. Finally, we thank Dan Berwick-Evans, Gunnar Carlsson, Daniel Litt, Gereon Quick, Kirsten Wickelgren, and particularly Sander Kupers for their helpful questions and insights. 

\section{Definitions and explanation of main statements}

We now define our notion of symmetric power in the algebraic and topological categories. On the algebraic side, given an algebraic space $X$ of finite type over a field $k$, we form the product $X^n$ over the base field $k$ together with its natural symmetric group $S_n$-action via permutation of the factors. We define $\Sym^n X$ as the quotient $X^n/S_n$, and here we really mean the ``naive'' quotient in the category of algebraic spaces as opposed to a stacky quotient (alternatively, take the coarse moduli space of the stacky quotient). We use the category of algebraic spaces in this paper to ensure that this quotient does indeed exist as it may not exist in the category of schemes.

In the topological category, given a space $X$, we again define $\Sym^n X = X^n / S_n$ as the ``naive'' quotient in the category of topological spaces (as opposed to a homotopy quotient or a quotient as a topological stack). The main comparison result that $\Sym^n (X_{et}) \stackrel{~}{\to} (\Sym^n X)_{et}$ is a weak equivalence holds true if we use the stacky quotient in the algebraic category and the homotopy quotient in the topological category, but this result is trivial to establish and does not have consequences for the stability phenomena we disuss in this paper. 

We use the \'{e}tale homotopy theory of Artin and Mazur in~\cite{AM} to extract homotopy types from algebraic spaces in order to formulate the comparison theorem. To be precise, Artin and Mazur construct a pro-object in the homotopy category of simplicial sets, but we will simply use the equivalent homotopy category of CW complexes for ease of discussion. Friedlander notably refined the \'{e}tale realisation functor to take values directly in pro-simplicial sets, but we do not need this level of sophistication here. We now briefly describe the \'{e}tale homotopy type construction. Given a suitably nice (locally noetherian) object $X$ with an \'{e}tale topology, we formally build some category from its covers such that the nerve of the category sees the underlying homotopy theoretic information $X$ in the \'{e}tale topology. The construction is entirely analogous to Borsuk's theorem, which says that a sufficiently nice space is homotopy equivalent to its Cech nerve. Furthermore, over $\mb{C}$, the idea that the \'{e}tale topology sees almost the same information as the classical topology except for only being able to make covers of finite degree still holds. More precisely, if $X$ is over the base field $\mb{C}$, denote by $X^{top}$ the topologification of $X$, i.e the $\mb{C}$-points of $X$ equipped with their classical topology. Then, if $X$ is normal, the \'{e}tale realisation $X_{et}$ is up to weak equivalence the profinite completion of the topological space $X^{top}$ by $12.10$ of~\cite{AM}.

\begin{rem} Given Friedlander's refinement of the \'{e}tale topological type to be valued in pro-simplicial sets, we could ask for an isomorphism in that category rather than the weaker question of weak equivalence of pro-homotopy types. We also do not address this question. \end{rem}

Now that we have acknowledged that the elements of our topological category are really pro-homotopy types rather than topological spaces, we reexamine our definition of the symmetric power functor. First, note that the symmetric power functor descends to the level of homotopy types: if $f : X \to Y$ is an equivalence of CW complexes, the natural morphism $\Sym^n f : \Sym^n X \to \Sym^n Y$ is also an equivalence. Indeed, using that an equivalence of CW complexes is an actual homotopy equivalence comprised of a pair of maps $f, g$, the fact that $\Sym^n$ is a functor allows us to take $\Sym^n f, \Sym^n g$ as a pair of maps exhibiting the homotopy equivalence of $\Sym^n X$ and $\Sym^n Y$. As such, we may talk about the symmetric power functor on homotopy types. Next, we define the symmetric power functor on pro-homotopy types ``levelwise'', i.e. if $\{X_{\alpha}\}$ is an pro-system for $X$, we define $\Sym^n X$ as given by the system $\{ \Sym^n X_{\alpha} \}$; all our results extend to this pro-category by functoriality.

G. Carlsson asked a natural question generalising Theorem $1$: suppose $X$ is a proper, geometrically connected algebraic space with an action of a finite group $G$. Is it in general true that the natural map $X_{et}/G \to (X/G)_{et}$ is a weak equivalence, perhaps under some mild hypotheses? For example, perhaps we take $X$ normal and suppose that the $G$-action on $X$ has a fixed point. The first observation is that we need more technology in order to even make sense of the question. By functoriality, the \'{e}tale realisation $X_{et}$ has a $G$-action, but if we view $X_{et}$ as only being a (pro-)homotopy type, it is impossible to take the quotient by the $G$-action in the desired way (as we are not interested in the homotopy quotient). As such, we have to use Friedlander's refinement of viewing $X_{et}$ directly as a pro-simplicial set; alternatively, we could use an equivariant refinement of \'{e}tale homotopy to view $X_{et}$ as a pro-object of the homotopy category of G-spaces. Both of these approaches have technical challenges which we are able to avoid in this paper due to the simple fact that the symmetric power construction respects homotopy equivalence. A further benefit to the particular case we consider here is that the homology of a symmetric power depends only on the homology groups of the original space, as originally shown by Dold in~\cite{Dold}. This sort of result makes it especially convenient to prove our main result via the Whitehead theorem, where we independently establish the result for the fundamental group and for homology. Hence, that the group action is $S_n$ acting by permutation on $X^n$ is extremely convenient and possibly necessary for the result. Returning to G. Carlsson's question, as posed there are immediate counterexamples such as $G = \mb{Z}/2$ acting by complex conjugation on $\Spec \mb{C}$. The intended form of the question therefore has the additional hypothesis that the $G$-action takes place over the base field, in which case we are unsure of the answer. Our best result, in general characteristic and not localised away from the characteristic of the field, is the one in this paper for the $S_n$ action on $X^n$ or at best something slightly stronger where $G$ is a sufficiently large subgroup of $S_n$.

Finally, we use the notion of weak equivalence for the category of pro-homotopy types from Definition $4.2$ of~\cite{AM}, which by Lemma $4.4$ there is equivalent to the usual notion of a map which induces isomorphisms on all homotopy (pro-)groups.  

\section{Aside: Motivation from the Weil conjectures}

Here we briefly turn to the number-theoretic importance of symmetric powers. Recall that the Weil conjectures for a smooth projective variety $X$ over a finite field $\mb{F}_q$ are typically phrased in terms of a generating function of the number of rational points of $X$ defined over the finite extensions of $\mb{F}_q$ via $$\zeta_X(t) = \exp \Big( \sum_{m=1}^{\infty} \frac{1}{m} |X(\mb{F}_{q^m})| t^m \Big),$$ where $|X(\mb{F}_{q^m})|$ denotes the number of rational points of $X$ over the field extension $\mb{F}_{q^m}$. However, by an elementary argument, we may re-express the zeta function above as follows: $$\zeta_X(t) = \sum_{n=0}^{\infty} |(\Sym^n X)(\mb{F}_q)| t^n.$$ We may now further use the Grothendieck-Lefschetz trace formula to express these point counts in terms of Frobenius traces on $\ell$-adic cohomology via the general formula $$|(\Sym^n X)(\mb{F}_q)| = \mathrm{sTr}(\Frob|H_c^*(\Sym^n X \otimes \overline{\mb{F}}_q, \mb{Q}_{\ell})),$$ where the $\mathrm{sTr}$ is the usual signed sum of traces by degree. The Weil conjectures are hence statements about (the Frobenius eigenvalues on) the topology of the sequence of spaces $\Sym^n X$. We now sketch how the Weil conjectures almost imply that $X$ connected is equivalent to stabilisation of the cohomologies of $\Sym^n X$, i.e. $$H^* \Sym^n X \stackrel{\sim}{\to} H^* \Sym^{n+1} X$$ for $* \ll n$, where we still mean $\mb{Q}_{\ell}$-cohomology of the base change to $\overline{\mb{F}}_q$ throughout. We take $X$ proper and thereby ignore the distinction between cohomology and compactly-supported cohomology. This argument uses the full strength of the Weil conjectures and does not even manage to prove the desired result (instead only speaking to the stabilisation of the Frobenius traces, e.g. completely ignoring subspaces on which Frobenius traces to zero). Nonetheless, we find it an interesting perspective for our topological stabilisation result.

First, $X$ connected is equivalent to a single factor of $(1 - t)$ in the denominator of $\zeta_X(t)$, with all other poles of the meromorphic function $\zeta_X(t)$ occurring at smaller values. As usual, ``smaller'' refers to the $q$-adic norm. By standard recurrence relation techniques, the fact that the largest pole is at $1$ means that, if we consider $\zeta_X(t)$ as a generating function for the Frobenius traces on $H^* \Sym^n X$, then these Frobenius traces are asymptotically constant. Of course, this observation is not surprising as the Frobenius trace on $H^* \Sym^n X$ for all $* > 0$ is already $q$-adically smaller than the contribution from the $H^0$ term by the Riemann hypothesis. However, something interesting happens if we instead consider the rational function $(1 - t) \zeta_X(t)$, which is now a generating function of the first finite difference of the Frobenius traces on $H^* \Sym^n X$. The fact that we had a single factor of $(1 - t)$ in the denominator of $\zeta_X(t)$ now means that the largest pole is $q$-adically at most $q^{1/2}$. Hence, the finite differences of Frobenius traces grow asymptotically like $q^{n/2}$ at most. Suppose now the cohomology groups did not stabilise, so that for some $j$, $H^j \Sym^n X$ continues to change for arbitrarily large $n$. Then, when we take the generating function of the finite difference (in $n$) of the Frobenius trace on the cohomology, that difference has a contribution from the $H^j \Sym^n X$ term for arbitrarily large $n$. However, the Riemann hypothesis says that the contribution to the Frobenius trace from the $H^j \Sym^n X$ term should be of order $O(q^{j/2})$, in contrast to the fact that these finite differences are supposed to be becoming arbitrarily small as $n$ increases. Thus, we have our supposed contradiction. Of course, there is not really any contradiction -- a priori, the Frobenius eigenvalues could conspire among themselves so that although each individually has magnitude $q^{j/2}$, they cancel in such a way that their sum does manage to become arbitrarily small as $n$ increases. For an extreme example of this phenomenon, we have absolutely no control over Frobenius-modules with trace zero. Hence, this argument is quite weak if interpreted literally as a heuristic that the Weil conjectures imply connected schemes have stable cohomology of symmetric powers, especially as showing stabilisation for rational cohomology is particularly easy. We instead take the perspective that from the viewpoint of the Weil conjectures, rational cohomological stabilisation is a slight topological strengthening of the presence of a single factor of $(1 - t)^{-1}$ in $\zeta_X(t)$. This paper then further strengthens that rational cohomological statement to an (integral) homotopical one. We hence regard the picture we present here of the stabilisation maps $\Sym^n X \to \Sym^{n+1} X$ becoming arbitrarily highly connected as a geometric explanation for some of the structure of the zeta function of $X$. This perspective immediately motivates the question: what are geometric explanations for other factors of $\zeta_X(t)$? What might a similar strengthening of the $(1 - qt)^{-1}$ factor in $\zeta_{\mb{P}^1}(t)$ look like? Cohomologically, we now expect some general statements on metastability; do these statements have corresponding refinements on the level of homotopy types?

\section{Stabilisation in the topological category}

We recall the arguments for the topological version of Corollary $4$ to set the stage for the algebro-geometric arguments to come in the next sections. We also note now that Corollary $4$ in the algebraic category follows immediately from the main comparison result of Theorem $1$ and the topological stabilisation result we establish in this section in Theorem $7$. We consider $(X, x)$ a pointed space homotopic to a CW complex. Of course, we really need to work in the category of pro-homotopy types, but as the symmetric power functor descends to the level of homotopy types and we define the symmetric power of a pro-homotopy type levelwise, one simply applies the below result ``levelwise'' on homotopy-types to derive the result for pro-homotopy types. Finally, we take as the base point in $\Sym^n X$ the image of the point $(x, x, \cdots, x) \in X^n$, denoted simply as $x$ by abuse of notation. Here is the general form of the stabilisation result we want:

\begin{thm} Given a pointed connected CW complex $(X, x)$, consider for any $n \ge 0$ the natural morphism $\alpha_n: \Sym^n X \to \Sym^{n+1} X$ given by including an extra copy of the base point $x$. Then $\pi_k (\alpha_n)$ is a surjection for $n = k$ and an isomorphism for $n > k$. \end{thm}

To see that the given range is strict, at least for $n$ odd, we claim that for $C$ a smooth proper complex curve of genus $g$, the natural map $\pi_{2g-1} \Sym^{2g-1} C \to \pi_{2g-1} \Sym^{2g} C$ is not an isomorphism. As $C$ is a normal algebraic variety, this counterexample will also demonstrate strictness in our algebraic version of the result; indeed, all the homotopy groups simply become profinitely completed. To establish our claim, note that $\Sym^{2g-1} C$ is a $\mb{P}^{g-1}$-bundle over $\Jac C$ while $\Sym^{2g} C$ is a $\mb{P}^g$-bundle over $\Jac C$. For simplicity, we suppose here that $g \ge 3$ so that, as $\Jac C$ is a $K(\pi, 1)$, the long exact sequence in homotopy for a fibration yields that $\pi_{2g-1} \Sym^{2g-1} C \simeq \pi_{2g-1} \mb{P}^{g-1} \simeq \pi_{2g-1} S^{2g-1} \simeq \mb{Z}$ and $\pi_{2g-1} \Sym^{2g} C \simeq \pi_{2g-1} \mb{P}^g \simeq \pi_{2g-1} S^{2g+1} = 0$. Hence the map is not an isomorphism. 

We now work our way up to showing this theorem. We first show that fundamental groups of symmetric powers behave as advertised by establishing the following statement:

\begin{prop} Given $(X, x)$ as above, for any $n \ge 2$, the natural morphism $$\pi_1(X, x) \to \pi_1(\Sym^n X, x)$$ factors through $\pi_1(X, x)^{ab}$ and induces an isomorphism $$\pi_1(X, x)^{ab} \stackrel{\sim}{\to} \pi_1(\Sym^n X, x).$$ \end{prop}

\begin{proof} Consider the quotient morphism $q: X^n \to \Sym^n X$. We claim that not only does $\pi_1$ applied to this morphism yield a surjection, but in fact every based loop $\gamma: (S^1, *) \to (\Sym^n X, x)$ lifts to a based loop (on the nose, not just up to homotopy): $$\xymatrix{ & (X^n, x) \ar[d]^{\pi} \\ (S^1, *) \ar@{-->}[ur] \ar[r]^{\gamma} & (\Sym^n X, x)}$$ Indeed, first we think of $S^1$ as $I / \partial I$, where $I$ is the interval, and we try to lift the map from $I$. Now, $q$ is naturally a map of stratified spaces, where the locally-closed strata are given by partitioning $\{1, ..., n\}$ into subsets and demanding that coordinates in the same subset are equal (and that coordinates in different subsets are distinct). Restricted to any stratum, $q$ is a covering space; for example, on the open stratum, we have a Galois covering space with Galois group $S_n$ while on the closed stratum (the diagonal, with all coordinates equal), we have the trivial covering space of degree $1$. Taking the preimage of this stratification under $\gamma$, we obtain a stratification of $I$ into locally closed subsets (i.e. collections of intervals), and we simply proceed from left to right, lifting along each interval by virtue of the fact that restricted to each stratum, $q$ really is a covering map. This procedure allows us to get a lift $I \to X^n$, but now it is straightforward to note that both endpoints are mapped to the same point as $x \in \Sym^n X$ has a unique lift, so we have in fact produced the desired lift $(S^1, *) \to (X^n, x)$. 

The above argument certainly implies that $\pi_1(q): \pi_1(X^n, x) \to \pi_1(\Sym^n X, x)$ is a surjection, but we have the canonical isomorphism $(\pi_1(X, x))^n \simeq \pi_1(X^n, x)$ induced by $\prod \pi_1(\iota_i)$, where $\iota_i$ is the inclusion $X \to X^n$ into the $i$th factor (given by the base point in other factors). As such, $$\prod \pi_1(q \circ \iota_i) : (\pi_1(X, x))^n \to \pi_1(\Sym^n X, x)$$ is a surjection. However, $q \circ \iota_i : X \to \Sym^n X$ is the same map for all $i$, so the map $\pi_1(X, x) \to \pi_1(\Sym^n X, x)$ induced by this one map must already be a surjection. As the map $X \to \Sym^n X$ is the composition $\alpha_{n-1} \circ \alpha_{n-2} \circ \cdots \circ \alpha_1$, we have that $$\pi_1 (\alpha_{n-1}) : \pi_1(\Sym^{n-1} X, x) \to \pi_1(\Sym^n X, x)$$ is also a surjection. We can go farther: if we let $\tau : X^2 \to X^2$ be the map given by switching the factors, and denoting once again $q: X^2 \to \Sym^2 X$, we have $q \tau = q$. Hence, given elements $\gamma_1, \gamma_2 \in \pi_1(X, x)$, we have \begin{align*} \pi_1(\alpha_1)(\gamma_1) \cdot \pi_1(\alpha_1)(\gamma_2) &= \pi_1(q) (\gamma_1 \otimes \gamma_2) \\ &= \pi_1(q \circ \tau) (\gamma_1 \otimes \gamma_2) \\ &= \pi_1(q) \pi_1(\tau)(\gamma_1 \otimes \gamma_2) \\ &= \pi_1(q) (\gamma_2 \otimes \gamma_1) \\ &= \pi_1(\alpha_1)(\gamma_2) \cdot \pi_1(\alpha_1)(\gamma_1). \end{align*} We already knew that $\pi_1(\alpha_1) : \pi_1(X, x) \to \pi_1(\Sym^2 X, x)$ was a surjection, but by the above, the map factors through $\pi_1(X, x)^{ab}$. In other words, we have a sequence of surjections $$\pi_1(X, x)^{ab} \twoheadrightarrow \pi_1(\Sym^2 X, x) \twoheadrightarrow \pi_1(\Sym^3 X, x) \twoheadrightarrow \cdots,$$ but by Dold-Thom and Hurewicz, we know the limiting value $$\colim \pi_1(\Sym^n X, x) \simeq \pi_1 (\Sym^{\infty} X, x) \simeq \tilde{H}_1(X; \mb{Z}) \simeq \pi_1(X, x)^{ab}.$$ Note that the colimit of a diagram of surjections is certainly surjected upon by any term of the colimit. Furthermore, we claim (again, from the definitions of the Dold-Thom and Hurewicz morphisms) that the composite map $$\pi_1(X, x)^{ab} \to \pi_1(\Sym^{\infty} X, x) \simeq \pi_1(X, x)^{ab}$$ is the identity, so in fact all the surjections considered above must be isomorphisms, as desired. \end{proof}

This part of the proof will transfer well to the algebro-geometric setting; of course, there, we are prohibited from thinking about $\pi_1$ in terms of loops, but the essential statement that $\pi_1(q)$ is a surjection due to its deeply ramified nature (in particular, the base point having a single lift) shall survive.

To proceed with the rest of the proof, we essentially have two choices: we can either proceed ``homotopically'' and directly prove that the maps on homotopy groups in low degrees are isomorphisms, or we could use a quantitative form of Whitehead's theorem which allows us to reduce homotopy isomorphisms to homology isomorphisms and an isomorphism on the fundamental group (which has already been taken care of). More precisely, we have the following: 

\begin{prop} Given a morphism of topological spaces or homotopy types $f: X \to Y$, the following two sets of conditions are equivalent: \\ \indent (i) $\pi_1 f$ is an isomorphism and $H_i(f, \mc{L})$ is an isomorphism for any $i < n$ and a surjection for $i = n$ for any local system $\mc{L}$. \\ \indent (ii) $\pi_i f$ is an isomorphism for $i < n$ and a surjection for $i = n$. \end{prop}

\begin{rem} Here, in the statement of the first condition, once we have the isomorphism of fundamental groups, we are using that isomorphism to identify local systems $\mc{L}$ on $X$ and $Y$. \end{rem}

Note that in the homological hypothesis, we really do need twisted coefficient systems in general, as for example the proof may proceed via a Serre spectral sequence computation of the homotopy fibre of $f$ to show it is sufficiently highly-connected, and the relevant homology groups that show up are indeed local systems on the base in general. Of course, if all spaces involved are simply-connected, then this quantitative version of Whitehead's theorem tells us that homological isomorphisms (with just integral coefficients, even, by the universal coefficient theorem) up to a certain degree immediately provides homotopical isomorphisms up to that degree. 

The above points are the essential ones that will translate well into the algebro-geometric case we treat next. To recap, the general strategy of the proof is to use a Whitehead-type result to split into showing stability for the fundamental group and stability for twisted (co)homology. Moreover, stability for the fundamental group should follow at least in principle from the highly ramified nature of $X^n \to \Sym^n X$. We now leave the rest of the proof of Theorem $7$ to the appendix for interested readers.

\section{Background and reductions}

We now move to the algebro-geometric side of the story. First, we comment on extending the definition of the \'{e}tale homotopy realisation functor to the category of algebraic spaces, rather than just the category of schemes. The definition of \'{e}tale homotopy type as given in $(9.6)$ of~\cite{AM} is set up to work very generally; the definition simply requires a locally connected topos, so under mild conditions such as the usual local noetherianness hypothesis on our algebraic space, we may define its \'{e}tale realisation using the \'{e}tale topos for the algebraic space in question. However, although it is perhaps only a philosophical point, we may worry whether our original \'{e}tale homotopy theory for schemes compares well to the \'{e}tale homotopy theory for algebraic spaces. In other words, if $X$ is a scheme, we already defined $X_{et}$ via \'{e}tale hypercovers built out of simplicial schemes, but now we think of $X$ as an algebraic space and instead consider its \'{e}tale homotopy type $X_{et}$ defined via the larger category of \'{e}tale hypercovers built out of simplicial algebraic spaces. Do the two notions agree? The latter category can be strictly larger than the former as not all \'{e}tale morphisms are representable; in other words, we can have an \'{e}tale morphism $\mc{X} \to X$ with $X$ a scheme but $\mc{X}$ a non-scheme algebraic space, as in tag 03FN of~\cite{Stacks}. Fortunately, separated, quasifinite, finite type morphisms are automatically representable by the quasi-affinity guaranteed by Zariski's main theorem for algebraic spaces, as in tag 05W7 of~\cite{Stacks}. In particular, the only possible issue is nonseparatedness of the \'{e}tale morphism; as long as we can guarantee separatedness, our morphism is representable. Of course, it is trivially true that any morphism of algebraic spaces is \'{e}tale-locally (on the source) separated, so any hypercover is dominated by one where all morphisms are separated and hence representable. Thus if $X$ is a scheme, its hypercovers by simplicial schemes are cofinal among its hypercovers by simplicial algebraic spaces and so the two notions of \'{e}tale homotopy type do in fact coincide. 

For ease of citing later theorems, we find it useful here to reduce from the case of a separably closed base field to an algebraically closed base field. This observation follows immediately from the following lemmas:

\begin{lemma} Given a field extension $k \hookrightarrow K$, we have a natural isomorphism $\Sym^n (X_K) \xrightarrow{\sim} (\Sym^n X)_K$. \end{lemma}

\begin{rem} Note that the notation above is slightly misleading in that the symmetric power functors are not quite the same: on the left side, the symmetric powers are over $\Spec K$ while on the right, they are over $\Spec k$.\end{rem}

\begin{proof} We have that $(X_K \times_K \cdots \times_K X_K) \simeq (X \times \cdots \times X)_K$ via a canonical isomorphism that intertwines the $S_n$-actions, so the (coarse moduli spaces of the) quotients are also canonically isomorphic. \end{proof}

\begin{lemma} For $X$ a locally noetherian algebraic space over a separably closed field $k$ and choice of algebraic closure $k \hookrightarrow \overline{k}$, the induced morphism $(X_{\overline{k}})_{et} \to X_{et}$ is a weak equivalence. \end{lemma}

\begin{proof} As $k \hookrightarrow \overline{k}$ is a purely inseparable extension, the corresponding map on spectra is radicial, in addition to being integral and surjective; the same properties hold for any base-change and we may apply topological invariance of the \'{e}tale site as in tag 05ZG of~\cite{Stacks}. \end{proof}

\section{Back to the fundamental group}

Using the above reduction, we henceforth suppose that $k$ is algebraically closed. As foreshadowed in the topological story, we first prove the commutativity result for the fundamental group. In other words, we show that the natural map $f: \pi_1(\Sym^n (X_{et}), x) \to \pi_1((\Sym^n X)_{et}, x)$ is an isomorphism. The case $n = 1$ is automatic, so suppose $n \ge 2$. What we want to show may be summarised using the following commutative diagram: \begin{equation*} \xymatrix{ & \pi_1(X_{et}, x) \ar[d] \ar[ddl] \ar[ddr] \\ & \pi_1(X_{et}, x)^{ab} \ar[dl]_{\sim} \ar@{-->}[dr] \\ \pi_1(\Sym^n (X_{et}), x) \ar[rr]^{f} & & \pi_1((\Sym^n X)_{et}, x). } \end{equation*}

From the topological theory, we already know that the indicated map $$\pi_1(X_{et}, x)^{ab} \to \pi_1(\Sym^n(X_{et}), x)$$ is an isomorphism, so to show that $f$ is an isomorphism, it suffices to show that the right map factors through the abelian quotient and that the resulting map is an isomorphism. Once we make some general notes on the fundamental group, we shall prove this claim in two steps: first, showing that the map factors as indicated and is a surjection and second, showing that the map is an injection. In the larger scheme of things, it is the first half of this statement that really controls stabilisation; once we have surjectivity, our original desired motivation of homotopical stabilisation of symmetric powers immediately follows. Of course, there are some downfalls: we would not know the stable value or exactly when it stabilises, and for technical reasons we would only know the stability prime-by-prime. Showing the second half of the statement, i.e. the injectivity, allows for a far cleaner statement. 

Recall how the fundamental group of the \'{e}tale realisation compares to the \'{e}tale fundamental group of SGA $1$~\cite{SGA1}: according to $10.7$ of~\cite{AM}, for any noetherian scheme $X$, $\pi_1(X_{et}, x)$ is the enlarged fundamental group of SGA 3.X.6~\cite{SGA3}, but the profinite completion of the enlarged fundamental group is always the usual \'{e}tale fundamental group of SGA $1$. However, by $11.1$ of~\cite{AM}, for $X$ a normal pointed connected noetherian scheme, the \'{e}tale realisation $X_{et}$ is already profinite so that in particular, its fundamental group is already profinite. Hence as long as we work with normal schemes (or more generally, geometrically unibranch schemes) as we are doing here, there is no distinction between the notions of fundamental group. Note also that dealing with algebraic spaces rather than schemes does not change the above argument: if we define the fundamental group $\pi_1$ as in SGA $1$ via the category of finite \'{e}tale covers of our algebraic space, this $\pi_1$ still agrees with the $\pi_1$ of the \'{e}tale realisation of the algebraic space. Indeed, a geometrically unibranch algebraic space still has profinite \'{e}tale realisation as the proof of the relevant result $11.1$ in~\cite{AM} immediately passes to a generic point of the scheme anyway and so perfectly well applies to algebraic spaces. Hence, in the following, we denote by $\pi_1$ the \'{e}tale fundamental group and refer to SGA $1$ for relevant results as they immediately generalise to geometrically unibranch algebraic spaces.

\subsection{Surjectivity}

\begin{prop} For $(X, x)$ as in Theorem 1, $\pi_1 \alpha_n$ is a surjection for all $n$. Moreover, all the maps factor through $\pi_1(X, x)^{ab}$, i.e. we have a chain of surjections $$\pi_1(X)^{ab} \to \pi_1 \Sym^2 X \to \cdots \to \pi_1 \Sym^n X \to \cdots.$$ \end{prop}

\begin{lemma} Let $f: X \to Y$ be a finite, generically \'{e}tale morphism between normal algebraic spaces. Then $[\pi_1(Y) : \im \pi_1(X)]$ divides $\deg f$ with equality if and only if $f$ is \'{e}tale. \end{lemma}

\begin{proof} Let $V \subset Y$ be an open subset over which $U = f^{-1}(V)$ is \'{e}tale; we take as base points for all our fundamental groups some geometric point of $U$ and its images in the other spaces but suppress them from the notation.

We now recall some basic statements in covering space theory, starting with the fundamental correspondence of covering space theory: finite $\pi_1(X)$-sets are in correspondence with finite \'{e}tale covers of $X$, where the orbits of the $\pi_1(X)$-action correspond to the various connected components of the cover and the association in one direction takes a connected finite \'{e}tale cover $\tilde{X} \to X$ and returns the cokernel of $\pi_1(\tilde{X}) \to \pi_1(X)$ as a $\pi_1(X)$-set. In particular, given any morphism of spaces $X_1 \to X_2$, the map $\pi_1(X_1) \to \pi_1(X_2)$ is a surjection if and only if every connected finite \'{e}tale cover of $X_2$ remains connected upon pullback to $X_1$. 

We have the following commutative diagram: \begin{equation*} \xymatrix{ 1 \ar[r] & \pi_1(U) \ar[r] \ar[d] & \pi_1(V) \ar[r] \ar[d] & S \ar[r] \ar[d] & 1 \\ & \pi_1(X) \ar[r] & \pi_1(Y) \ar[r] & T \ar[r] & 1 } \end{equation*} Here, the diagram is in the category of (pointed) sets. We claim $\pi_1(V) \to \pi_1(Y)$ is surjective, i.e. that the pullback of any connected finite \'{e}tale cover of $Y$ to $V$ remains connected. Suppose then that we have some \'{e}tale cover $\tilde{Y} \to Y$ that disconnects upon restriction to $\tilde{V} \to V$; we must show $\tilde{Y}$ was originally disconnected. First, note that an \'{e}tale cover of a normal space remains normal, as in tags 033C and 033G of~\cite{Stacks}. (Or rather, we define normality of algebraic spaces by using that normality may be verified \'{e}tale-locally, as per the lemmas.) Hence $\tilde{Y}$ must be the normalisation of $Y$ in the function field extension $k(\tilde{Y})= k(\tilde{V})$, but by assumption, $\tilde{V}$ is disconnected and so $k(\tilde{V})$ is a nontrivial direct sum of fields, so that the normalisation of $Y$ in this extension must also be disconnected, as desired. (See tag 0823 of~\cite{Stacks} to confirm that normalising algebraic spaces in function field extensions makes sense and is unique.) We will continue to use this idea of constructing \'{e}tale covers of normal spaces by using normalisation inside function field extensions in the remainder of the proof.

As $U \to V$ is \'{e}tale, the exact sequence $1 \to \pi_1(U) \to \pi_1(V) \to S \to 1$ arises from covering space theory and in particular, $|S| = \deg f$; meanwhile, $|T|$ is essentially by definition $[\pi_1(Y) : \im \pi_1(X)]$ and so the surjectivity of $S \to T$ already gives the divisibility in the statement of the lemma. Indeed, if we consider $S \to T$ as a map of sets, we a priori only have an inequality, but $S \to T$ is a morphism of transitive $\pi_1(V)$-sets, and in particular the transitivity yields that all the fibres of the morphism have the same cardinality. Now, if $f$ is \'{e}tale, we would again have an exact sequence $1 \to \pi_1(X) \to \pi_1(Y) \to T \to 1$ with $|T| = \deg f$. The harder direction is the converse, so suppose we have $|S| = |T| = \deg f$, i.e. that the surjection $S \to T$ is in fact an isomorphism. Using the fundamental correspondence of covering space theory, we form the connected finite \'{e}tale cover $\tilde{X} \to Y$ associated to the transitive $\pi_1(Y)$-set $T$. We can run exactly the above argument again, forming the exact sequence $1 \to \pi_1(\tilde{U}) \to \pi_1(V) \to \tilde{S} \to 1$, where $\tilde{U} = V \times_Y \tilde{X}$. We find that $|\tilde{S}| = |T| = \deg f$ yields $\tilde{S} \to T$ an isomorphism so in fact $S$ and $\tilde{S}$ are isomorphic as $\pi_1(V)$-sets. Using the fundamental correspondence again, their associated finite \'{e}tale covers $\tilde{U}$ and $U$ coincide. We hence have $U$ sitting as a dense open inside both $X$ and $\tilde{X}$, and as both $X$ and $\tilde{X}$ are normal, the uniqueness of normalisation in function field extensions yields that $X = \tilde{X}$ so that $X$ is \'{e}tale, as desired. \end{proof}

We remark that this approach of using normality hypotheses to pin down a finite cover once we know it on a dense open should be able to yield more information. For example, we would like to extend the purview of the lemma so that it also implies that if $f$ is totally ramified at some point, then $\pi_1(X) \twoheadrightarrow \pi_1(Y)$; in general, we might surmise that $[\pi_1(Y) : \im \pi_1(X)]$ divides the cardinality of any fibre of the map (where we really mean the cardinality of the reduction, as opposed to the length, which would always return the degree). 

\begin{proof} We now prove Proposition $18$. We would like to check that the hypotheses of the lemma apply, so let us first note that $X$ normal implies $\Sym^i X$ normal for all $i$, as products of normal algebraic spaces and quotients of normal algebraic spaces by finite groups exist and are normal, by Theorem $4.3$ of~\cite{Kolquot}. As usual, we take as base points for the fundamental groups of the $\Sym^i X$ the point consisting of $i$ repeated copies of $x$. These base points shall be clearly compatible with all maps between the symmetric powers that we consider; we shall suppress them from the notation. We shall throughout use the Kunneth theorem of SGA 1.X.1.7~\cite{SGA1} that if $X, Y$ are proper noetherian pointed schemes over an algebraically closed field, $\pi_1(X \times Y) \to \pi_1(X) \times \pi_1(Y)$ is an isomorphism. It is precisely at this step citing the Kunneth theorem that the necessity of properness as a hypothesis enters the argument.

We briefly remark that Grothendieck's Kunneth theorem extends to the case of (proper noetherian pointed) algebraic spaces: the main technical input is the use of Stein factorisation, but the formation of the Stein factorisation by taking a relative Spec and its basic properties from the theorem of formal functions all work for proper morphisms of noetherian algebraic spaces. Indeed, the main idea of the proof of the Kunneth formula is that if $X \to Y$ is a proper, separable morphism with $Y$ (locally) noetherian and connected, then if $y$ is a geometric point of $Y$, we have the right-exact sequence $\pi_1(X_y) \to \pi_1(X) \to \pi_1(Y) \to 1$. The main content of exactness in the middle is that an \'{e}tale cover $X' \to X$ is pulled back from an \'{e}tale cover of $Y$ if its base-change to $X_y$ is trivial (splits completely), so under these suitable hypotheses, there has to be a construction somehow of an \'{e}tale cover of $Y$: this step is taken care of exactly by (the second map in) the Stein factorisation of the composite morphism $X' \to Y$. So, as advertised, this technology continues to work in the category of algebraic spaces, and to check the necessary properties, e.g. that this second map is indeed \'{e}tale over $Y$, we can simply perform an \'{e}tale base-change (which functorially preserves the Stein factorisation) to make everything into a scheme so as to appropriate the proofs from SGA $1$. 

Now, we have a chain of maps $X \to \Sym^2 X \to \Sym^3 X \to \cdots$. We show by induction on $n$ that $\pi_1 X^n \to \pi_1 \Sym^n X$ is surjective. For $n = 1$, there is nothing to show, so consider the general case wherein we already know the statement for all smaller $n$. If $n = p$ is prime, consider the degree $p$ finite map $X \times \Sym^{p-1} X \to \Sym^p X$ between normal spaces; as it is certainly \'{e}tale over the locus of distinct points and is certainly not \'{e}tale over the locus where points coincide, Lemma $19$ applies to give us that $[\pi_1(\Sym^p X) : \im \pi_1(X \times \Sym^{p-1} X)]$ strictly divides $p$, i.e. $\pi_1 (X \times \Sym^{p-1} X) \to \pi_1 \Sym^p X$ is a surjection. Using the inductive hypothesis, we have $$\pi_1 X^p \simeq \pi_1 X \times \pi_1 X^{p-1} \twoheadrightarrow \pi_1 X \times \pi_1 \Sym^{p-1} X \simeq \pi_1 (X \times \Sym^{p-1} X) \twoheadrightarrow \pi_1 \Sym^p X.$$ Next, consider the case that $n$ is composite, and consider all pairs $(i, j)$ of positive integers with $i + j = n$. The map $X^n \to \Sym^n X$ factors as $X^n \to \Sym^i X \times \Sym^j X \to \Sym^n X$ and the inductive hypothesis yields $$\pi_1 X^n \simeq \pi_1 X^i \times \pi_1 X^j \twoheadrightarrow \pi_1 \Sym^i X \times \pi_1 \Sym^j X$$ so we see $$[\pi_1 \Sym^n X : \im \pi_1 X^n] = [\pi_1 \Sym^n X : \im \pi_1(\Sym^i X \times \Sym^j X)] \: \Big| \: \binom{n}{i},$$ where we used Lemma $19$ and that observation that the degree of $\Sym^i X \times \Sym^j X \to \Sym^n X$ is $\binom{n}{i}$. Now, it is an easy consequence of Kummer's theorem that $\gcd_{1 \le i \le n} \binom{n}{i}$ is $p$ if $n = p^k$ is a prime power and $1$ otherwise. In particular, if $n$ is not a prime power, we already have $\pi_1 X^n \twoheadrightarrow \pi_1 \Sym^n X$. If $n = p^k$ is a prime power with $k > 1$, note $X^{p^k} \to \Sym^{p^k} X$ also factors as $X^{p^k} \to \Sym^p \Sym^{p^{k-1}} X \to \Sym^{p^k} X$. As $\Sym^{p^{k-1} X}$ is also a space satisfying the hypotheses of Proposition $18$, the inductive hypothesis yields $$\pi_1 X^{p^k} \simeq \pi_1 (X^{p^{k-1}})^p \twoheadrightarrow \pi_1 (\Sym^{p^{k-1}} X)^p \twoheadrightarrow \pi_1 \Sym^p \Sym^{p^{k-1}} X,$$ and therefore $$[\pi_1 \Sym^{p^k} X : \pi_1 X^{p^k}] = [\pi_1 \Sym^{p^k} X : \pi_1 \Sym^p \Sym^{p^{k-1}} X] \: \Big| \: \frac{p^k!}{p! (p^{k-1}!)^p}.$$ As $\val_p p^k! = \frac{p^k - 1}{p-1}$, the integer on the RHS above is coprime to $p$ and we once again establish the desired surjectivity.

The rest of the argument proceeds as in the topological version of section $4$. We have shown $(\pi_1 X)^n \simeq \pi_1(X^n) \twoheadrightarrow \pi_1 \Sym^n X$ for all $X$, but the $n$ maps $\pi_1 X \to \pi_1 \Sym^n X$ are all the same map. Indeed, they are induced by the maps $X \to X^n \to \Sym^n X$, where the first map is the inclusion into the $i$th coordinate (with the base point $x$ in all other coordinates), followed by the quotient; of course, this quotient identifies all $n$ of these maps. Hence the image of $(\pi_1 X)^n$ is the same as the image of $\pi_1 X$, so we have established the surjectivity of $\pi_1 X \twoheadrightarrow \pi_1 \Sym^n X$. Finally, we may consider the commutative diagram \begin{equation*} \xymatrix{ X \times X \ar[rr] \ar[dr] & & X \times X \ar[dl] \\ & \Sym^2 X } \end{equation*} where the horizontal map switches the two factors. The diagram commutes, and upon applying $\pi_1$, we find that the map $\pi_1 X \to \pi_1 \Sym^2 X$ naturally factors through $(\pi_1 X)^{ab}$. These considerations suffice to establish the proposition. \end{proof}

\subsection{Injectivity}

\begin{prop} For $X$ as above and all $n \ge 2$, the natural map $\pi_1(X, x)^{ab} \to \pi_1(\Sym^n X, x)$ is an isomorphism. \end{prop}

\begin{proof} For technical reasons, we want to replace $X$ with $X^{red}$. We claim that the symmetric power of a reduced space is also reduced: recall that the product of reduced spaces over an algebraically closed field remains reduced. Moreover, a geometric quotient of a reduced space is also still reduced as (\'{e}tale) locally, we are simply considering a ring of invariants, which certainly does not have any nilpotents if the original ring does not. Finally, Theorem $1.5$ of~\cite{Kolquot} implies that finite group quotients are indeed geometric quotients, so that symmetric powers of reduced spaces are reduced. Hence, replacing $X$ by $X^{red}$ would simultaneously replace $\Sym^n X$ by $\Sym^n X^{red} \simeq (\Sym^n X)^{red}$. We now recall that topological invariance of the \'{e}tale site, as in tag 05ZG of~\cite{Stacks}, implies as a special case that for any algebraic space $X$, $\pi_1(X^{red}) \stackrel{\sim}{\to} \pi_1(X)$ (in fact, the entire \'{e}tale homotopy type is unaffected). Hence to establish this proposition (or indeed any of the other theorems in this paper), we may throughout assume $X$ reduced.

By the surjectivity of $\pi_1(X, x)^{ab} \twoheadrightarrow \pi_1(\Sym^n X, x)$, we already know that $$\{\text{finite }\pi_1(\Sym^n X, x)\text{-sets}\} \to \{\text{finite }\pi_1(X, x)^{ab}\text{-sets}\}$$ is injective; if we could show that this map is a bijection, that would imply the desired result. Translating into the language of finite \'{e}tale covers, we find that we must find a way to lift every abelian finite \'{e}tale cover of $X$ to a finite \'{e}tale cover of $\Sym^n X$ (i.e. such that we obtain the original \'{e}tale cover under pullback along the canonical map $X \to \Sym^n X$.) We will now summarise the ensuing argument by the following diagram, to be gradually explained: \begin{equation*} \xymatrix{ & & Z_1 \simeq Z_2 \ar[ddl] \ar[d] \\ & \tilde{V} \ar[ddl] \ar[r] & Y \ar[ddl] \\ U \ar[r] \ar[d] \ar[drr] & X^n \ar[r] \ar[d] \ar[dr] & BA^n \ar[d] \\ V \ar[r] \ar@/_1pc/[rr] & \Sym^n X \ar@{.>}[r] & BA } \end{equation*}

Let's start reading the diagram with the bottom $2 \times 3$ rectangle. Denote by $U \subset X^n$ the complement of the ``fat diagonal'', i.e. the open locus where no coordinates coincide, and let $V \subset \Sym^n X$ be the image of $U$. Now, $U \to V$ is an \'{e}tale map (indeed, a Galois map with Galois group $S_n$). So, suppose now that we have an abelian finite \'{e}tale cover $\tilde{X} \to X$ from which we wish to produce a cover of $\Sym^n X$; of course, we may suppose that the covering space is in fact Galois with Galois group some finite abelian group $A$. Then the cover is equivalent to a map to the classifying stack $X \to BA$. (Of course, all of the following may be rephrased without using the technology of stacks, but the idea is clearer in this language.) Then, recalling that we have $(BA)^n \simeq B(A^n)$ and any morphism of groups $G \to H$ yields a morphism $BG \to BH$, the fact that $A$ is abelian means the addition map $A^n \to A$ is a group morphism and so we have a map $BA^n \to BA$. Hence let us form the composite map $X^n \to BA^n \to BA$, which represents an abelian \'{e}tale cover $Z_1 \to X^n$. We wish to try to descend this cover to $\Sym^n X$, but $X^n \to \Sym^n X$ is not flat, so we cannot directly descend. However, if we restrict to $U \hookrightarrow X^n \to BA$, then $U \to V$ is a categorical quotient; hence, as the map $U \to BA$ is invariant under the $S_n$-action on $U$, the map factors to provide a map $V \to BA$. We again rephase this map in terms of a finite \'{e}tale cover $\tilde{V} \to V$. Now, we normalise $\Sym^n X$ inside the function field extension $k(\tilde{V})$ to obtain a normal space equipped with a finite map $Y \to \Sym^n X$ which restricts to $\tilde{V}$ over $V$. We have that $Y \times_{\Sym^n X} X^n$ is a normal algebraic space $Z_2$ with a finite morphism $Z_2 \to X^n$. However, $Z_1$ and $Z_2$ agree over $U \subset X^n$ and so are both the normalisation of $X^n$ in the common function field extension, so by uniqueness of this construction, $Z_1$ and $Z_2$ are the same finite cover of $X^n$. Hence the finite map $Y \to \Sym^n X$ pulls back to an \'{e}tale map under the surjective map $X^n \to \Sym^n X$. We now use the following lemma (where we could certainly drop some of the hypotheses). Note that $X^n \to \Sym^n X$, as an example of the quotient of an algebraic space by a reductive group, is an affine morphism by Theorem $3.12$ of~\cite{Kolquot} and so in particular is finite for the current case of a finite group quotient. 

\begin{lemma} Suppose that we have a pullback diagram of locally noetherian algebraic spaces \begin{equation*} \xymatrix{ \tilde{X} \ar[d] \ar[r] & X \ar[d] \\ \tilde{Y} \ar[r] & Y} \end{equation*} where $\tilde{X} \to X$ is \'{e}tale, $X \to Y$ is finite surjective, $\tilde{Y} \to Y$ is finite, and $Y$ is reduced. Then $\tilde{Y} \to Y$ is \'{e}tale. \end{lemma}

\begin{proof} It suffices to check \'{e}tale-locally that a morphism is \'{e}tale, so base-change the entire diagram along an \'{e}tale atlas of $Y$ with schematic source. We may hence assume that all algebraic spaces in the above diagram are in fact schemes. Now, recall that one way to check flatness of a finite morphism over a reduced base is constancy of the Euler characteristic in fibres by, for example, 24.7.A.(d) of~\cite{Vakil}. By surjectivity of $X \to Y$, all fibres of $\tilde{Y} \to Y$ appear among the fibres of $\tilde{X} \to X$, so as $\tilde{X} \to X$ is flat so that its fibres have constant Euler characteristic, the constancy is certainly true for the fibres of $\tilde{Y} \to Y$. We hence have the desired flatness. Similarly, all geometric fibres of $\tilde{Y} \to Y$ appear amongst the geometric fibres of $\tilde{X} \to X$, so recalling that one definition of \'{e}tale is to be locally of finite presentation, flat, and have all geometric fibres a disjoint union of the base, we see that this property certainly holds for $\tilde{Y} \to Y$ given that it does for $\tilde{X} \to X$. \end{proof}

We have now produced a finite \'{e}tale cover $Y \to \Sym^n X$, which we would like to check yields the original cover of $X$ that we started with, but the natural map $X \to \Sym^n X$ factors through $\iota_1: X \to X^n$, which we recall is given by $x' \mapsto (x', x, x, \cdots, x)$ where $x$ is the base point of $X$. Hence as $Y \to \Sym^n X$ was produced by descent of $Z_1 \to X^n$, which was in turn produced via the composite map $X^n \to BA^n \to BA$, we simply need to check that the composite $X \stackrel{\iota_1}{\to} X^n \to BA^n \to BA$ yields the original map. This verification involves retranslating into the language of \'{e}tale covers; recall that if $\tilde{X} \to X$ is a $G$-torsor on $X$ and we have a morphism $G \to H$, one forms the associated $H$-torsor by the quotient $(\tilde{X} \times H) / G$ where we use the diagonal action of $G$. Hence in this situation, the cover of $X$ represented by this composite map is $((\tilde{X} \times A^{n-1}) \times A) / A^n$, where $A^n$ acts on the last factor $A$ by addition of all coordinates and $A^n$ acts on the first factor $\tilde{X} \times A^{n-1}$ by the first $A$ in $A^n$ acting on $\tilde{X}$ and the remaining $A^{n-1}$ acting by translation on the $A^{n-1}$ factor. By a shearing change-of-coordinates, we may rewrite this action as $(\tilde{X} \times A^n) / A^n$, where $A^n$ simply acts on the $A^n$ factor by translation, so that the quotient is the original cover $\tilde{X} \to X$, as desired. \end{proof}

\section{Whitehead and Deligne}

We now have the fundamental group statement of the desired commutativity relation, i.e. that $\pi_1(\Sym^n (X_{et}), x) \to \pi_1((\Sym^n X)_{et}, x)$ is an isomorphism. It remains to show that the commutativity also holds for cohomology with all local systems so that we can then conclude by citing an appropriate version of Whitehead's theorem. Theorem $4.3$ of ~\cite{AM} is this appropriate version of the Whitehead theorem, which we reproduce here in the special case of interest:

\begin{thm} Let $f: X \to Y$ be a morphism of pro-finite pro-homotopy types. Then the following are equivalent: \\ \indent (i) The morphism $f: X \to Y$ is a weak equivalence. \\ \indent (ii) The induced map $\pi_1(X) \to \pi_1(Y)$ is an isomorphism and for every twisted abelian finite coefficient group $M$, $H^q(Y, M) \to H^q(X, M)$ is an isomorphism for all $q$. \\ \indent (iii) THe induced map $\pi_1(X) \to \pi_1(Y)$ is an isomorphism and for every induced map of finite (Galois) covering spaces $X' \to Y'$ and every (untwisted) finite abelian coefficient group $A$, we have $H^q(Y, A) \to H^q(X, A)$ an isomorphism for all $q$. \end{thm}

The third condition is a special case of the second as cohomology of a topological cover is the same as twisted cohomology with coefficients in the $\mb{Z}[\pi_1]$-module given by $A[\pi_1 / N]$. Here, $A$ is the original coefficient group, $N$ is the subgroup of $\pi_1$ corresponding to the cover, and $A[\pi_1/N]$ is the free $A$-group on generators $\pi_1/N$ (if $N$ is normal, so that $\pi_1/N$ is a group, this construction is the group ring construction) with $\pi_1$-action given by acting on the generators. As such, the third condition is saying that rather than needing all twisted coefficient systems, we can make do with those of the form $A[\pi_1/N]$, where $N$ is some finite-index subgroup. As stated in~\cite{AM}, the third condition asks for all finite covering spaces, but in fact only Galois covering spaces arise in the reduction from (ii) to (iii). This difference will not matter for us as we shall work with abelian fundamental groups, but we find it psychologically easier to know one can initially reduce to the Galois case.

Finally, by $11.1$ of~\cite{AM}, as all algebraic spaces we are working with are pointed, connected, normal (hence geometrically unibranch), and noetherian, their \'{e}tale realisations are pro-finite. Hence, the above Whitehead theorem applies and we turn to verifying the necessary cohomological commutativity result. Note that by $9.3$ and $10.8$ of~\cite{AM}, the cohomology with twisted coefficients of the \'{e}tale homotopy type agrees with \'{e}tale cohomology with coefficients in a local system as usually defined via the \'{e}tale topos (on either a scheme or an algebraic space; the added complexity of algebraic spaces makes no difference to the categorical constructions happening here). Hence for checking the commutativity for cohomology, we may use the standard theory of \'{e}tale cohomology.

We use Deligne's computation of the cohomology of symmetric powers in SGA 4.XVII.5.5.21 of~\cite{SGA4}. Conveniently, Deligne computes cohomology with coefficients in a quite general class of torsion sheaves (and even a derived category thereof), which is exactly the generality we need to apply the Whitehead result above. In fact, we claim that Deligne's result from SGA 4 is precisely an algebraic analogue of Dold's computation of the homology of symmetric powers in~\cite{Dold}, with some added generality (Deligne uses more general coefficient sheaves and works over a general base). Indeed, the general claim of this Dold-Deligne formula is that the (co)homology of a symmetric power is a derived symmetric power of the (co)homology of the original space. For a reader directly reading Dold's paper, this conclusion may be opaque. To provide some orienting remarks, Dold writes his formula in terms of symmetric powers of a resolution by FD-modules. In fact, Dold's FD-modules are simplicial abelian groups, and where one uses resolutions by chain complexes to derive abelian functors, the appropriate resolutions for deriving nonabelian functors such as the symmetric power functor is by simplicial objects. Hence, Dold's formula is precisely a model for the derived symmetric power functor and therefore compares perfectly to Deligne's formula. We now reproduce a simplified statement of the Deligne result $5.5.21$ of SGA 4.XVII:

\begin{thm} Let $f: X \to \Spec k$ be a locally noetherian separated algebraic space, $\mc{A}$ a torsion commutative ring, $n$ a nonnegative integer, and $K$ an object in $D^{b, tor \le 0}(X, \mc{A})$. Then the symmetric Kunneth morphism $$ L\Gamma^n_{ext} Rf_! K \to R \Sym^n(f)_! L\Gamma^n_{ext} K$$ is an isomorphism. \end{thm}

Before we explain the notation in the statement of this theorem, we note that Deligne states his result for a quasiprojective scheme. The quasiprojectivity restriction is there so that $\Sym^n X$ exists in the category of schemes and hence may be removed if we are willing to work throughout in the category of algebraic spaces. Furthermore, Deligne indicates that the theorem should easily generalise to the category of algebraic spaces in remark $5.5.21.1$, and indeed, one of the first reductions Deligne performs in the proof of this result is to reduce via d\'{e}vissage to the case that $K$ is shriek-pushed forward from an affine locally closed subscheme of $X$. We may simply appropriate the d\'{e}vissage result he proves in $5.5.22.2$ and recall that a separated algebraic space enjoys a locally finite stratification into locally closed schematic subspaces to reduce to the same situation. In other words, we may once again suppose that $X$ is an affine scheme and proceed with Deligne's proof. 

We now finally explain the notation in the statement of the theorem. Here $D^{b, tor \le 0}(X, \mc{A})$ is some version of the bounded derived category of $\mc{A}$-module sheaves on $X$ with a tor-dimension bound, $\Gamma^n_{ext}$ is the $n$th symmetric product functor, $L \Gamma^n_{ext}$ is the derived functor thereof, and $Rf_!$ is total cohomology with compact support. Fortunately, we may simplify this statement in the particular case we are interested in: first, for us, $X$ is proper, so that $f_! = f_*$ and $Rf_*$ is the total cohomology $R \Gamma$ (hopefully, this notation will not clash too much with the notation used for the derived symmetric product functor $L \Gamma^n _{ext}$). Next, we take $K$ to simply be a single sheaf in degree zero consisting of some $\mc{A}$-local system. In particular, this choice of $K$ certainly satisfies criterion $5.5.13.1$ of~\cite{SGA4} that for all geometric points $x \in X$, $K_x$ is (homotopic to) a complex of flat $\mc{A}_x$-modules, as for us, $K_x$ may be identified with $A_x$, so in particular, $5.5.14$ applies and $L \Gamma^n_{ext} K$ is just $\Sym^n K$ defined as a descent of the $S_n$-invariance subsheaf of the external box power $K^{\boxtimes n}$. As such, we make a further simplification of Deligne's statement to give the following:

\begin{thm} Let $X$ be a proper algebraic space, $\mc{A}$ a torsion commutative ring, $n$ a nonnegative integer, and $K$ an $\mc{A}$-local system, i.e. an \'{e}tale bundle of groups with fibre $\mc{A}$. Then the symmetric Kunneth morphism $$ L \Gamma^n_{ext} R\Gamma K \to R\Gamma (\Sym^n K)$$ is an isomorphism. \end{thm}

We shall continue to not explain exactly what $L \Gamma^n_{ext}$ is, accepting it as some explicit prescription for computing the total cohomology $R\Gamma (\Sym^n X, \Sym^n K)$ out of the original cohomology $R\Gamma (X, K)$. Note that (in agreement with Dold) exactly the same prescription holds in the topological case where now we do not even need to take compactly-supported cohomology for our base-change theorems used in the proof to hold, i.e. if $K$ were a topological local system on a topological space $X$ rather than an \'{e}tale local system on a scheme or algebraic space $X$: indeed, Deligne's proof in the algebraic category proceeds by reduction to the topological case anyway, and as Deligne explains in the introduction to $5.5$, the proposition in the topological category may be directly verified on the cochain level. 

Now, let us prove the cohomological criterion we need in the pro-Whitehead result to show that $\Sym^n (X_{et}) \to (\Sym^n X)_{et}$ is a weak equivalence, i.e. we need to show that for all $q$ and all twisted coefficient groups of the form $A[\pi_1 \Sym^n X / N]$, we have $$H^q((\Sym^n X)_{et}, A[\pi_1 \Sym^n X / N]) \to H^q(\Sym^n (X_{et}), A[\pi_1 \Sym^n X / N])$$ is an isomorphism. We may of course suppose that $n \ge 2$, so that $\pi_1 \Sym^n X$ may be identified with $\pi_1 X^{ab}$. We recall that for torsion sheaves, the cohomology of the \'{e}tale realisation agrees with \'{e}tale cohomology. We apply the Dold-Deligne formula in both the topological and algebraic categories where we take $\mc{A}$ to be the torsion commutative ring $A[\pi_1 X^{ab} / N]$ with its group ring structure and the local system $K$ on $X$ to be the $\mb{Z}[\pi_1 X]$-module $A[\pi_1 X^{ab}/N]$, where $\pi_1(X)$ acts in the evident way on the generators. (We use the isomorphism between the \'{e}tale fundamental group of $X$ as an algebraic space and the fundamental group of $X_{et}$ to speak equivalently about topological local systems on $X_{et}$ and \'{e}tale local systems on $X$.) Now, we have that $\Sym^n K$, as defined above, is precisely the original local system $A[\pi_1 \Sym^n X / N]$, and so we have the following chain of isomorphisms: \begin{eqnarray*} R \Gamma ((\Sym^n X)_{et}, A[\pi_1 \Sym^n X / N]) & \stackrel{\sim}{\to} & R\Gamma(\Sym^n X, A[\pi_1 \Sym^n X / N]) \\ & \stackrel{\sim}{\to} & L\Gamma^n_{ext} R\Gamma(X, K) \\ & \stackrel{\sim}{\to} &  L \Gamma^n_{ext} R\Gamma(X_{et}, K) \\ & \stackrel{\sim}{\to} & R\Gamma (\Sym^n (X_{et}), A[\pi_1 \Sym^n X / N]. \end{eqnarray*} This completes the cohomological comparison and thus, by the pro-Whitehead theorem, the proof of the weak equivalence.

\section{Appendix: Proof of stability in the topological case}

We now finish the proof of Theorem $7$ from Section $4$.

\begin{proof} Theorem $22.7$ of~\cite{Steenrod} or Corollary $5.2$ of~\cite{Milgram} already shows homological stability for the symmetric powers $\Sym^n X$, with Milgram providing explicit descriptions of these homology groups. As such, if $X$ is simply-connected, our Proposition $9$ implies that all the symmetric powers are simply-connected and so Steenrod's result suffices to show homotopical stability. For general $X$, one can follow Steenrod's or Milgram's analysis and verify that the homological isomorphisms they find hold for twisted coefficient systems as well. Instead, in this appendix, we provide an argument that is more explicitly homotopical and, we find, provides more geometric insight for why stabilisation occurs.

We need an important lemma as a technical input for our homotopical argument. Consider the representation $V$ of the symmetric group $S_n$ formed as a direct sum of $k$ copies of the standard $n$-dimensional representation. Pick an equivariant metric and consider the unit sphere $S^{nk-1}$ inside $V$, together with its inherited $S_n$-action. We then claim the following and will return to its proof later:

\begin{lemma} For the $S_n$-action on $S^{nk-1}$ as above, the quotient $S^{nk-1}/S_n$ is $(n-1)$-connected. \end{lemma}

So consider now some connected CW complex $X$. We look at the natural decomposition we get for $\Sym^n X$ given a CW decomposition for $X$; first, as $X$ is connected, we choose a CW decomposition which contains a single $0$-vertex. Now, from a CW decomposition for $X$, we do not quite get a CW decomposition for $\Sym^n X$: instead, we get a prescription for building up $\Sym^n X$ from products of symmetric powers of cells (rather than from cells only). Indeed, let $e_i$ denote the cells of $X$, with $e_i$ homeomorphic to $D^{d_i}$ for some $d_i$ and attached along the boundary $S^{d_i - 1}$. We fix $e_0$ to be the unique $0$-cell. Now, we naturally have a decomposition of $\Sym^n X$ where we iteratively glue on spaces of the form $\Sym^{n_0} e_0 \times \Sym^{n_1} e_1 \times \cdots \times \Sym^{n_{\ell}} e_{\ell}$, where $\sum_{i=1}^{\ell} n_i = n$ (the $n_i$ are allowed to vanish). More precisely, denote by $(\Sym^n X)_r$ the intermediate spaces we get in this process of gradually building up $\Sym^n X$; what we mean by gluing on these spaces means that we iteratively form pushouts (which are homotopy pushouts, as the inclusion of the boundary is a cofibration) as follows: \begin{equation*} \xymatrix{ \partial (\Sym^{n_0} e_0 \times \cdots \times \Sym^{n_{\ell}} e_{\ell}) \ar[r] \ar[d] & (\Sym^n X)_r \ar[d] \\ \Sym^{n_0} e_0 \times \cdots \times \Sym^{n_{\ell}} e_{\ell} \ar[r] & (\Sym^n X)_{r+1}. } \end{equation*} The homotopy type of a symmetric power only depends on the homotopy type of the original space, so the actual space $\Sym^{n_0} e_0 \times \cdots \times \Sym^{n_{\ell}} e_{\ell}$ we are gluing on is still contractible, so that in forming the homotopy pushout above, we are essentially just coning off its boundary. We hence want to understand the homotopy type of this boundary. Now, in general, if $Y$ and $Z$ are both contractible, then the boundary of their product is homotopic to the join of their boundaries: $$\partial (Y \times Z) = (Y \times Z) \setminus (Y^{\circ} \times Z^{\circ}) = (Y \times \partial Z) \underset{\partial Y \times \partial Z}{\cup} (\partial Y \times Z) \simeq (\partial Y) \star (\partial Z).$$ We thus need to understand the boundary of a symmetric product of a cell, but provided that $k \ge 1$, we have $$\partial \Sym^n D^k \simeq S^{nk-1}/S_n,$$ where the group action is exactly as described in the setup to Lemma $24$. Hence this boundary $\partial \Sym^n D^k$ is at least $(n-1)$-connected. As the join of an $(n-1)$-connected and $(m-1)$-connected space is $(n+m)$-connected, we see that provided $n_0 = 0$ (so that we only consider products of symmetric powers of positive-dimensional cells), we have that $\partial(\Sym^{n_0} e_0 \times \cdots \times \Sym^{n_{\ell}} e_{\ell})$ will be at least $(\sum n_i) - 1 = (n-1)$-connected. Therefore, coning off this space will not affect any homotopy groups in degree $(n-1)$ or less and can only induce a surjection on $\pi_n$. 

Now, consider the inclusion $\Sym^n X \hookrightarrow \Sym^{n+1} X$; the latter is obtained from the former by gluing on products of symmetric powers of cells. More precisely, a product of symmetric powers of cells $\Sym^{n_0} e_0 \times \cdots \Sym^{n_{\ell}} e_{\ell}$ in the decomposition for $\Sym^{n+1} X$ is also in the copy of $\Sym^n X$ sitting inside $\Sym^{n+1}X$ if and only if $n_0 \ne 0$. In other words, the new products of symmetric powers which we are gluing on have $n_0 = 0$. Hence the connectivity bound from the previous paragraph applies and we have the desired stability result. \end{proof}

We now return to the proof of Lemma $24$. In fact, the lemma is equivalent to the homotopical stabilisation of Theorem $7$ for the special case of symmetric powers of the sphere $S^k$. Indeed, the space $S^{nk-1}/S_n$ is exactly the cofibre of $\Sym^{n-1} S^k \hookrightarrow \Sym^n S^k$ (which is also a homotopy cofibre, as the map is a cofibration). Hence, for $k > 1$, we could claim that the lemma already follows from Steenrod's result as $S^k$ is simply-connected. (As all symmetric powers of $S^1$ are homotopy equivalent to $S^1$, the result is particularly trivial for $k = 1$). However, we provide our own argument for Steenrod's homological stabilisation result in this special case. Steenrod's clever algebraic decomposition of the chains of a symmetric power is certainly shorter, but we like our more pedestrian analysis for the special case of the sphere, which has the side benefit of keeping this part of the paper self-contained.

\begin{proof} We throughout use the long exact sequence for Borel-Moore homology that if $X$ has a closed subspace $Y$ with open complement $U$, we have $$ \cdots \to H_i^{BM}(Y) \to H_i^{BM}(X) \to H_i^{BM}(U) \to H_{i-1}^{BM}(Y) \to \cdots.$$ For example, consider $S^k$ as the one-point compactification of $\mb{R}^k$ with base point at $\infty$ so that $\Sym^n S^k \setminus \Sym^{n-1} S^k$ is exactly the collection of points with no point at the base point at $\infty$, i.e. $\Sym^n \mb{R}^k$. Alternatively, we can describe our space of interest $S^{nk-1}/S_n$ as the one-point compactification of $\Sym^n \mb{R}^k$. In any case, we see that it suffices to show $H_*^{BM} \Sym^n \mb{R}^k$ is trivial for $* < n$, which we show by induction on $k$, the case $k = 1$ being trivial. Now, in general, we consider the partial compactification of $\mb{R}^k$ to $\mb{R}^{k-1} \times I$, with $I$ denoting the unit interval $[0, 1]$. We have $\Sym^n \mb{R}^k = \Sym^n (\mb{R}^{k-1} \times I) \setminus B$, where we will describe $B$ shortly. We describe $\Sym^n (\mb{R}^{k-1} \times I)$ using the natural map $\pi: \Sym^n (\mb{R}^{k-1} \times I) \to \Sym^n I$, where we use the coordinatisation $\Sym^n I = \{(x_1, \cdots, x_n) | 0 \le x_1 \le \cdots \le x_n \le 1 \}$. This simplex has a natural finite stratification by vertices, edges, $2$-dimensional faces, and so on; over each locally closed stratum, $\pi$ is a product. For example, over the stratum of $\Sym^6 I$ given by $x_1 = x_2, x_4 = x_5 = x_6$, the fibre of $\pi$ is $\Sym^2 \mb{R}^{k-1} \times \mb{R}^{k-1} \times \Sym^3 \mb{R}^{k-1}$. In general, the fibres of $\pi$ will be products of symmetric powers of $\mb{R}^{k-1}$ so that by the Kunneth formula for Borel-Moore homology and the inductive hypothesis, the fibre of $\pi$ has no Borel-Moore homology in degrees smaller than $n$. We also note that $B$ is the preimage under $\pi$ of the two faces of the simplex $\Sym^n I$ given by $x_1 = 0$ and $x_n = 1$. We now claim that both $\Sym^n (\mb{R}^{k-1} \times I)$ and $B$ have no Borel-Moore homology in degrees smaller than $n$, which will establish the claim by the long exact sequence for Borel-Moore homology. However, these claims easily follow from the description of $\pi$ as having a product structure subordinate to a finite stratification with fibres having vanishing Borel-Moore homology in low degrees: simply iteratively use the long exact sequence for Borel-Moore homology. Hence the claim is established. \end{proof}

\bibliographystyle{alpha}
\bibliography{ref}

\newcommand{\etalchar}[1]{$^{#1}$}
\begin{thebibliography}{ABD{\etalchar{+}}64}

\bibitem[ABD{\etalchar{+}}64]{SGA3}
M.~Artin, J.~E. Bertin, M.~Demazure, P.~Gabriel, A.~Grothendieck, M.~Raynaud,
  and J.-P. Serre.
\newblock {\em Sch\'emas en groupes. {F}asc. 3: {E}xpos\'es 8 \`a 11}, volume
  1963/64 of {\em S\'eminaire de G\'eom\'etrie Alg\'ebrique de l'Institut des
  Hautes \'Etudes Scientifiques}.
\newblock Institut des Hautes \'Etudes Scientifiques, Paris, 1964.

\bibitem[AM69]{AM}
M.~Artin and B.~Mazur.
\newblock {\em Etale homotopy}.
\newblock Lecture Notes in Mathematics, No. 100. Springer-Verlag, Berlin, 1969.

\bibitem[Dol58]{Dold}
Albrecht Dold.
\newblock Homology of symmetric products and other functors of complexes.
\newblock {\em Ann. of Math. (2)}, 68:54--80, 1958.

\bibitem[Fri73]{Friedfibs}
Eric~M. Friedlander.
\newblock Fibrations in etale homotopy theory.
\newblock {\em Inst. Hautes \'Etudes Sci. Publ. Math.}, (42):5--46, 1973.

\bibitem[Gro63]{SGA1}
Alexander Grothendieck.
\newblock {\em Rev\^etements \'etales et groupe fondamental. {F}asc. {II}:
  {E}xpos\'es 6, 8 \`a 11}, volume 1960/61 of {\em S\'eminaire de G\'eom\'etrie
  Alg\'ebrique}.
\newblock Institut des Hautes \'Etudes Scientifiques, Paris, 1963.

\bibitem[Kol97]{Kolquot}
J{\'a}nos Koll{\'a}r.
\newblock Quotient spaces modulo algebraic groups.
\newblock {\em Ann. of Math. (2)}, 145(1):33--79, 1997.

\bibitem[Mil69]{Milgram}
R.~James Milgram.
\newblock The homology of symmetric products.
\newblock {\em Trans. Amer. Math. Soc.}, 138:251--265, 1969.

\bibitem[SGA73]{SGA4}
{\em Th\'eorie des topos et cohomologie \'etale des sch\'emas. {T}ome 3}.
\newblock Lecture Notes in Mathematics, Vol. 305. Springer-Verlag, Berlin,
  1973.
\newblock S{\'e}minaire de G{\'e}om{\'e}trie Alg{\'e}brique du Bois-Marie
  1963--1964 (SGA 4), Dirig{\'e} par M. Artin, A. Grothendieck et J. L.
  Verdier. Avec la collaboration de P. Deligne et B. Saint-Donat.

\bibitem[{Sta}13]{Stacks}
The {Stacks Project Authors}.
\newblock Stacks project.
\newblock \url{http://stacks.math.columbia.edu}, 2013.

\bibitem[Ste72]{Steenrod}
Norman~E. Steenrod.
\newblock Cohomology operations, and obstructions to extending continuous
  functions.
\newblock {\em Advances in Math.}, 8:371--416, 1972.

\bibitem[SV96]{VS}
Andrei Suslin and Vladimir Voevodsky.
\newblock Singular homology of abstract algebraic varieties.
\newblock {\em Invent. Math.}, 123(1):61--94, 1996.

\bibitem[Vak13]{Vakil}
Ravi Vakil.
\newblock Foundations of algebraic geometry.
\newblock \url{http://math.stanford.edu/~vakil/216blog/}, 2013.

\bibitem[Voe02]{Voevodsky}
Vladimir Voevodsky.
\newblock Motivic cohomology groups are isomorphic to higher {C}how groups in
  any characteristic.
\newblock {\em Int. Math. Res. Not.}, (7):351--355, 2002.

\end{thebibliography}
\end{document}